\documentclass[11pt, a4paper]{article}

\usepackage[top=120pt]{geometry}

\usepackage[utf8]{inputenc}
\usepackage{amsmath, amsthm, amsfonts, amssymb}
\usepackage{mathtools}

\usepackage{enumerate}

\usepackage[bookmarks]{hyperref}
\usepackage{indentfirst} 

\usepackage{authblk}
\usepackage{lipsum}

\usepackage[backend=biber,maxbibnames=99]{biblatex}
\usepackage{mleftright}

\newtheorem{lemma}{Lemma}[section]

\newtheorem{proposition}[lemma]{Proposition}
\newtheorem{corollary}[lemma]{Corollary}

\newtheorem{result}{Theorem}

\newtheorem*{main}{Main Theorem}

\theoremstyle{definition}

\newcommand{\abs}[1]{\ensuremath{\left| #1 \right|}}
\newcommand{\op}{\operatorname}
\newcommand{\ce}[2]{\operatorname{C}_{#1}(#2)}
\newcommand{\no}[2]{\operatorname{N}_{#1}(#2)}
\newcommand{\ze}[1]{\operatorname{Z}(#1)}

\newcommand{\rup}[2]{\op{O}^{#1}(#2)}
\newcommand{\syl}[2]{\op{Syl}_{#1}\left(#2\right)}

\newcommand{\groupgen}[1]{\langle #1 \rangle}

\newcommand{\irr}{\operatorname{Irr}}
\newcommand{\lin}{\operatorname{Lin}}

\newcommand{\aut}{\operatorname{Aut}}

\newcommand{\gal}{\operatorname{Gal}}
\newcommand{\cycl}[1]{\mathbb{Q}_{#1}}
\newcommand{\smid}{\! \mid \!}

\newcommand{\Q}{\mathbb{Q}}

\def\Aut{\mathrm{Aut}}

\renewcommand{\phi}{\varphi}
\renewcommand{\epsilon}{\varepsilon}

\DeclareMathSymbol{\sminus}{\mathbin}{AMSa}{"39}

\title{Odd degree rational characters and the order of rational elements in finite groups}

\date{\today}

\author{Nicola Grittini \thanks{The author is grateful to INPS - Istituto Nazionale di Previdenza Sociale for supporting him financially during his research.}}

\begin{document}

\maketitle

\begin{abstract}
We prove that, in a finite group, if every rational irreducible character has odd degree, then all rational elements are $2$-elements, as it was originally conjectured by P. H. Tiep and H. P. Tong-Viet.
\end{abstract}

\section{Introduction}

Let $G$ be a finite group. An irreducible character $\chi \in \irr(G)$ is said to be \textit{rational} if $\chi(x) \in \Q$ for each $x \in G$. On  the other hand, an element $x \in G$ is said to be \textit{rational} in $G$ if it is $G$-conjugate to every generator of the cyclic group $\groupgen{x}$, and the conjugacy class $x^G$ is said to be rational if $x$ is rational.

It is known that all the irreducible characters of a group $G$ are rational if and only if all the conjugacy classes of $G$ are rational, while the principal character is the only rational irreducible character of $G$ if and only if $G$ has odd order and, therefore, it does not have any rational elements other then the identity. Moreover, it is proved in \cite{Navarro-Tiep:Rational_irreducible_characters} that a finite group $G$ has exactly two rational irreducible characters if and only if it has exactly two rational conjugacy classes and, in this case, both rational irreducible characters have odd degree. On the other hand, it is easy to see that, if $G$ has exactly two rational conjugacy classes, then the non-trivial rational class needs to consist of elements of order $2$. Hence, it makes sense to conjecture, as it was done by P. H. Tiep and H. P. Tong-Viet in \cite{Tiep-TongViet:Odd-degree_rational}, that, if every rational irreducible character of a group $G$ has odd degree, then all rational elements of $G$ have order a power of $2$. In this paper, we prove this conjecture.

%It is conjectured that this equality holds also for three rational irreducible characters and three conjugacy classes, while there are known counterexamples for four rational irreducible characters.

%It is also proved in \cite{Navarro-Tiep:Rational_irreducible_characters} that, if $G$ has exactly two rational irreducible characters, then they are both of odd degree. On the other hand, it is easy to see that, if $G$ has exactly two rational conjugacy classes, then one needs to be of elements of order $2$. Hence, it makes sense to conjecture, as it was done by Tiep and Tong-Viet in \cite{Tiep-TongViet:Odd-degree_rational}, that, if every rational character of a group $G$ has odd degree, then all rational elements of $G$ have order a power of $2$. In this paper, we prove said conjecture.

\begin{main}
Let $G$ be a finite group. If every rational irreducible character of $G$ has odd degree, then all rational elements of $G$ are $2$-elements.
\end{main}

The conjecture was already proved in \cite{Tiep-TongViet:Odd-degree_rational} for a fairly large number of cases. In fact, it was proved that, if all the rational irreducible characters of a finite group $G$ have odd degree and no simple group $\op{PSL}_2(3^f)$, with $f$ odd, is involved in $G$, then $G$ is solvable. Then, the conjecture was proved in this solvable context. In this paper, we are able to generalize the approach adopted in \cite{Tiep-TongViet:Odd-degree_rational} for solvable groups, itself partially borrowed from \cite{Navarro-Tiep:Rational_irreducible_characters}, to the non-solvable case.

As it was already noted in \cite{Tiep-TongViet:Odd-degree_rational}, the converse of our Main Theorem does not hold, since the dihedral group $D_8$ has a rational irreducible character of degree $2$. On the other hand, however, the proof for the solvable case presented in \cite{Tiep-TongViet:Odd-degree_rational} implies that, if a group $G$ has a rational element of order $p$, then it has a rational $B_p$-character, according to the definition introduced by Isaacs in \cite{Isaacs:Characters_pi-separable_groups}. It should not be too difficult to prove that, in this case, the converse actually holds and, if $G$ has a rational $B_p$-character, then it has a rational $p$-element (and it is also fairly easy to see that a rational $B_p$-character, for $p$ odd, must be of even degree). Of course, the fact that $B_p$-characters are defined only for $p$-solvable groups strongly limits the extent to which we could apply this property. If, however, we can find a way to extend the definition of $B_p$-characters also for non-solvable groups, as our Lemma~\ref{lemma:existing_over} might suggest, then we may find a new interesting relation between rational characters and rational elements in finite groups.

\section{Results}

We first need to prove a result on the extension of characters in non-abelian factor groups. We begin with a proposition on \textit{character triples isomorphisms} (see \cite[Definition~11.23]{Isaacs}).

\begin{proposition}
\label{proposition:triples_isomorphism}
Let $(G,N,\phi)$ and $(\Gamma,A,\lambda)$ be isomorphic character triples and suppose that $A \leq \Gamma'$ and $\lambda$ is linear. Let $p$ be a prime such that $p \mid o(\lambda)$, then $p \mid \phi(1)o(\phi)$.
\end{proposition}

\begin{proof}
Suppose $p \nmid \phi(1)o(\phi)$, let $P/N$ be a Sylow $p$-subgroup of $G/N$ and let $Q/A$ be a Sylow $p$-subgroup of $\Gamma/A$ corresponding to $P/N$ under the isomorphism of character triples. Since $p \nmid \phi(1)o(\phi)$, we have by \cite[Corollary~6.28]{Isaacs} that $\phi$ extends to $\hat{\phi} \in \irr(P \smid \phi)$. Let $\hat{\lambda} \in \irr(Q \smid \lambda)$ corresponding to $\hat{\phi}$ under the isomorphism of character triples, then it follows from \cite[Lemma~11.24]{Isaacs} that $\hat{\lambda}(1)/\lambda(1) = \hat{\phi}(1)/\phi(1) = 1$ and, thus, $\lambda$ extends to $Q$. In particular, the $p$-part $\lambda_p \neq 1_A$ of $\lambda$ also extends to $Q$ and it follows from \cite[Theorem~6.26]{Isaacs} that $\lambda_p$ extends to $\Gamma$, in contradiction with the assumption that $A \leq \Gamma'$.
\end{proof}

For a finite group $S$, we will write $M(S)$ to denote the \textit{Schur multiplier} of $S$. The reader can see \cite[Definition~11.12]{Isaacs} for a definition of Schur multiplier.

\begin{corollary}
\label{corollary:extension_simple_groups}
Let $G$ be a finite group and let $N \lhd G$. Let $p$ be a prime and let $\phi \in \irr(N)$ be a $G$-invariant character such that $\phi(1)o(\phi)$ is a power of $p$. Suppose that $p \nmid \abs{M(G/N)}$, then $\phi$ extends to $G$.
\end{corollary}

Notice that the corollary holds also if we replace $p$ with a set of primes $\pi$.

\begin{proof}
Let $(\Gamma,A,\lambda)$ be a character triple isomorphic to $(G,N,\phi)$ such that $A \leq \ze{\Gamma}$, $A \leq \Gamma'$ and $\abs{A} = \abs{M(G/N)}$ (see \cite[Theorem~11.17, Corollary~11.20 and Theorem~11.28]{Isaacs}). Then it follows from Proposition~\ref{proposition:triples_isomorphism} that $\lambda=1_A$ and we have by \cite[Theorem~11.7]{Isaacs} that $\phi$ extends to $G$.
%Let $P/N$ be a Sylow $r$-subgroup of $G/N$, for some prime number $r$. If $r \neq p$, then $\phi$ extends to $P$ by \cite[Corollary~6.28]{Isaacs}. Suppose $r=p$ and $P/N$ is a Sylow $p$-subgroup of $G/N$, and let $(\Gamma,A,\lambda)$ be a character triple isomorphic to $(G,N,\phi)$ such that $\abs{A} = \abs{M(G/N)}$ (see \cite[Corollary~11.20 and Theorem~11.28]{Isaacs}). Let $Q/A$ be a Sylow $p$-subgroup of $\Gamma/A$ corresponding to $P/N$ under the isomorphism of character triples; since $p \nmid \abs{A}$, then $\lambda$ extends to $\hat{\lambda} \in \irr(Q \smid \lambda)$ by \cite[Corollary~6.28]{Isaacs}. Let $\hat{\phi} \in \irr(P \smid \phi)$ corresponding to $\hat{\lambda}$ under the character triple isomorphism, then it follows from \cite[Lemma~11.24]{Isaacs} that $\hat{\phi}(1)/\phi(1) = \hat{\lambda}(1)/\lambda(1) = 1$ and, thus, $\phi$ extends to $P$. Finally, it follows from \cite[Corollary~11.31]{Isaacs} that $\phi$ extends to $G$.
\end{proof}

\begin{corollary}
\label{corollary:extension_nonabelian_quotient}
Let $G$ be a finite group and let $N \lhd G$ such that $G/N \cong S_1 \times ... \times S_k$ for some non-abelian simple groups $S_i$, $i=1,...,k$. Let $p$ be a prime and suppose that $p \nmid \abs{M(S_i)}$ for each $i \in \{1,...,k\}$. Suppose that $\phi \in \irr(N)$ is $G$-invariant and that $\phi(1)o(\phi)$ is a power of $p$. Then, $\phi$ has an irreducible extension $\chi$ to $G$ such that $o(\chi)=o(\phi)$.
\end{corollary}

\begin{proof}
We prove it by induction on $k$. Let $N \leq S \lhd G$ such that $S/N \cong S_k$, then it follows from Corollary~\ref{corollary:extension_simple_groups} that $\phi$ has an extension $\hat{\phi} \in \irr(S)$. Since $S/N$ is perfect, $\hat{\phi}$ is the only extension of $\phi$ to $S$ and it follows that it is $G$-invariant. Moreover, $\det(\hat{\phi})$ is the unique extension to $S$ of $\det(\phi)$ and it follows that $o(\hat{\phi})=o(\phi)$. Hence, $\hat{\phi} \in \irr(S)$ is a $G$-invariant character and $\hat{\phi}(1)o(\hat{\phi})=\phi(1)o(\phi)$ is a power of $p$. Since $G/S \cong S_1 \times ... \times S_{k-1}$, it follows by induction that $\hat{\phi}$ has an extension $\chi \in \irr(G)$ such that $o(\chi)=o(\hat{\phi})=o(\phi)$, and $\chi_N=\hat{\phi}_N=\phi$.
\end{proof}

The following lemma is, de facto, an extension of \cite[Lemma~6.2]{Navarro-Tiep:Rational_irreducible_characters} for (some) non-solvable groups. Since we are not working with $p$-solvable groups, we cannot just use the properties of Isaacs' $B_p$-characters, as it is done in \cite{Navarro-Tiep:Rational_irreducible_characters}. However, we prove that some properties of those characters still hold in our non-solvable context. It may be also possible to extend the definition of $B_p$-characters to non-solvable groups, by eventually imposing some conditions, on the group structure, which are weaker then the $p$-solvability. This, however, would go too far beyond the scope of our paper.

We should also mention that a weaker result then the following, rather technical lemma may have been sufficient for our purposes as well. However, our lemma could be of interest also of its own, as an extension of \cite[Lemma~6.2]{Navarro-Tiep:Rational_irreducible_characters}, and for this reason the author decided to include it here.

We should first introduce some notation. For $n \mathbb{N}$, we write $\Q_n$ to denote the $n$th-cyclotomic extension of $\Q$, i.e., the extension of $\Q$ by a primitive $n$th-root of unity. Moreover, if $G$ is a group, $a \in \aut(G)$, $U \leq G$ and $\phi \in \irr(U)$, then we define $(U,\phi)^a = (U^a, \phi^a)$. Notice that, in this case, $\phi^a \in \irr(U^a)$.

\begin{lemma}
\label{lemma:existing_over}
Let $G$ be a finite group, let $p$ be an odd prime, let $M \lhd G$ and suppose that, for each subgroup $M \leq H \leq G$ and each non-abelian composition factor $S$ of $H/M$, $p \nmid \abs{M(S)}$. Let $\phi \in \irr(M)$ and suppose there exists $U \leq M$ and $\tau \in \irr(U)$ such that $\phi=\tau^M$ and $\tau(1)o(\tau)$ is a power of $p$. Moreover, suppose that, for any automorphism $b$ of $M$ fixing $\phi$, $(U,\tau)$ and $(U,\tau)^b$ are conjugate in $M$. Then, there exists a subgroup $U \leq W \leq G$ and a character $\theta \in \irr(W \smid \tau)$ such that $\chi = \theta^G$ is irreducible and lies over $\phi$, $\theta(1)o(\theta)$ is a power of $p$ and, for any automorphism $c$ of $G$ fixing $M$ and $\chi$, $(W,\theta)$ and $(W,\theta)^c$ are conjugate in $G$.

Moreover, let $\sigma \in \gal(\cycl{\abs{G}}/\mathbb{Q})$ of order $p-1$ such that it fixes $p'$th-roots of unity, and write $\xi^{\sigma}=\xi^k$ when $\xi$ is any $p$th-root of unity. Let $a$ be an automorphism on $G$ which leaves both $M$ and $U$ invariant. Assume that the order of $a$ on its action on $G/M$ is not divisible by $p$. If $\tau^a = \tau^{\sigma}$, then $a$ normalizes $W$, $\theta^a = \theta^{\sigma}$, and $\chi^a=\chi^{\sigma}$.

Finally, assume that there is no element $x \in G/E$ of order $p$ such that $x^{a^{\sminus 1}} = x^k$ for every $a$-invariant subgroup $E$ with $M \leq E \lhd G$. If $\tau$ has values in $\cycl{p}$, then also both $\theta$ and $\chi$ have values in $\cycl{p}$.
\end{lemma}

\begin{proof}
We prove the lemma by induction on $\abs{G:M}$. Let $T=I_G(\phi)$ be the inertia subgroup of $\phi$ in $G$ and suppose first that $T < G$. By induction, there exists $U \leq W \leq T$ and $\theta \in \irr(W \smid \tau)$ such that $\theta^T = \psi \in \irr(T \smid \phi)$, then $\chi=\psi^G=(\theta^T)^G$ is irreducible and it lies over $\phi$. In particular, if $\tau$ has values in $\cycl{p}$, then by induction so does $\theta$ and, therefore, so does $\chi$, too. Moreover, if $c$ acts as an automorphism on $G$ preserving $M$ and it fixes $\chi$, then $\phi^c=\phi^g$ for some $g \in G$, since both $\phi$ and $\phi^c$ are irreducible constituents of $\chi_M$. In particular, $x = c g^{\sminus 1}$ normalizes $M$ and fixes $\phi$. It follows that $T^x=T$ and, since both $\chi$ and $\phi$ are fixed by $x$, then $x$ also fixes $\psi$. It follows by induction that $(W,\theta)^x = (W,\theta)^h$ for some $h \in T$ and, therefore, $(W,\theta)^c = (W,\theta)^{hg}$.

Finally, suppose that $\tau^a = \tau^{\sigma}$, where $a$ is the automorphism of $G$ we fixed in the second paragraph of the lemma, then by induction also $\theta^a=\theta^{\sigma}$. Notice that, for $g \in G$,
$$ (\theta^a)^G(g) = \frac{1}{\abs{W}} \sum_{x \in G} \theta^{\circ}(x^{a^{\sminus 1}}g^{a^{\sminus 1}}(x^{a^{\sminus 1}})^{\sminus 1}) = \frac{1}{\abs{W}} \sum_{x \in G} \theta^{\circ}(x g^{a^{\sminus 1}} x^{\sminus 1}) = \theta^G(g^{a^{\sminus 1}}) $$
and, therefore, $(\theta^a)^G = (\theta^G)^a$. Since also $(\theta^{\sigma})^G = (\theta^G)^{\sigma}$, we have that $\chi^a=\chi^{\sigma}$.

From now on, we can assume that $\phi$ is $G$-invariant. Let $\mathcal{H} \leq \Aut(G)$ consist of all the automorphisms of $G$ leaving $M$ invariant, so that $a \in \mathcal{H}$, and let $M < L \lhd G$ be minimal such that $\mathcal{H}$ leaves $L$ invariant. Notice that $L/M$ is either an elementary abelian group or it is the direct product of some non-abelian simple groups.

Let $g \in L$. Then $\phi^g=\phi$, because $\phi$ is $G$-invariant, and it follows from the hypotheses that $(U,\tau)^g=(U,\tau)^h$ for some $h \in M$. Let $N=\no{L}{U}$ and $T=I_N(\tau)$, then we have that $g h^{\sminus 1} \in T$ and, thus, $g \in TM$. It follows that $L=TM$. Moreover, we have that $T \cap M = I_{\no{M}{U}}(\tau) = U$, since $\tau^{T \cap M}$ is irreducible, and it follows that $T/U \cong L/M$.

We claim that there exists $\gamma \in \irr(T \smid \tau)$ such that $\gamma(1)o(\gamma)$ is a power of $p$ and $\gamma^L = \psi \in \irr(L \smid \phi)$. We also claim that, if $d \in \mathcal{H}$ fixes $\psi$, then $(T,\gamma)$ and $(T,\gamma)^d$ are conjugate in $L$. Moreover, if $U^a=U$ and $\tau^a=\tau^{\sigma}$, we claim that also $T^a=T$ and $\gamma^a=\gamma^{\sigma}$. Finally, if $\tau$ has values in $\cycl{p}$ and there is no element $x \in L/M$ of order $p$ such that $x^{a^{\sminus 1}} = x^k$, we claim that $\gamma$ has also values in $\cycl{p}$.% Notice that, if $U \leq H \lhd T$, then $H/U \cong HM/M \lhd G/M$ and, thus, the hypothesis on the elements of $H/U$ are always verified whenever the analogous hypothesis is verified for the elements of $G/M$.

If $T/U \cong L/M$ is either a $p'$-group or the direct product of non-abelian simple groups, then $\tau$ extends to $T$ (see Corollary~\ref{corollary:extension_nonabelian_quotient} for the non-abelian case) and we can call $\gamma$ the only extension of $\tau$ to $T$ such that $o(\gamma)=o(\tau)$. We have that $[(\gamma^L)_M,\phi]=[(\gamma_U)^M,\phi]=[\tau^M,\phi]=1$, thus, $\phi$ is a constituent of $(\gamma^L)_M$. However, $\gamma^L(1)=\abs{L:T}\gamma(1)=\abs{M:U}\tau(1)=\phi(1)$ and it follows that $\gamma^L = \psi \in \irr(L)$ is an extension of $\phi$. If $d \in \mathcal{H}$ fixes $\psi$, then it also fixes $\phi$ and, by hypothesis, $(U,\tau)^d = (U,\tau)^g$ for some $g \in M$; it follows that $dg^{\sminus 1}$ normalizes $U$ and fixes $\tau$. In particular, $dg^{\sminus 1}$ also normalizes $T=I_N(\tau)$ and fixes $\gamma$ (because of the uniqueness properties of $\gamma$ as an extension of $\tau$). Hence, we have that $(T,\gamma)^d=(T,\gamma)^g$. Finally, if $a$ normalizes $U$ and $\tau^a = \tau^{\sigma}$, then $a$ also normalizes $T$, since $I_N(\tau)=I_N(\tau^{\sigma})$. By the uniqueness property, if $(\tau^{\sigma})^{a^{\sminus 1}} = \tau$, then also $(\gamma^{\sigma})^{ a^{\sminus 1}} = \gamma$ and, moreover, if $\tau$ has values in $\cycl{p}$, then so does $\gamma$. Hence, we have proved the claim in this case.

Assume now that $T/U \cong L/M$ is an abelian $p$-group and notice that, in this case, both the order and the degree of every irreducible constituent of $\tau^T$ are powers of $p$. By \cite[Lemma~2.2]{Wolf:Character_correspondences}, there exists a unique subgroup $E$ with $M \leq E \leq L$ maximal with respect to the property of $\phi$ having an $L$-invariant extension $\hat{\phi} \in \irr(E)$, and $\hat{\phi}$ is fully ramified in $L/E$. Since $L=TM$, there exists $U \leq H_1 \leq T$ such that $E=H_1M$. Moreover, if we apply \cite[Lemma~2.2]{Wolf:Character_correspondences} to $(T,U,\tau)$, we have that there exists a unique subgroup $H_2$ with $U \leq H_2 \leq T$ maximal with respect to the property of $\tau$ having a $T$-invariant extension $\hat{\tau} \in \irr(H_2)$, and $\hat{\tau}$ is fully ramified in $T/U$. Now, by proceeding as in the previous paragraph we have that $\hat{\tau}^{H_2M}$ lies over $\phi$ and, since $\hat{\tau}^{H_2M}(1) = \phi(1)$, it follows that $\hat{\tau}^{H_2M}$ extends $\phi$ to $H_2M$. Since $\hat{\tau}$ is $T$-invariant, then so is $\hat{\tau}^{H_2M}$, and it follows from the maximality of $E=H_1M$ that $H_2 \leq H_1$. On the other hand, we may notice that 
$$ [\hat{\phi}_U,\tau] = [\phi_U,\tau] = [\phi,\tau^M] = 1. $$
In particular, we can choose $\hat{\tau}$ such that it is the only irreducible constituent of $\hat{\phi}_{H_2}$ lying over $\tau$, since $[\hat{\phi}_{H_2},\tau^{H_2}] = [\hat{\phi}_U,\tau] = 1$, and we have that there exists $\eta \in \irr(H_1)$ such that it is the only irreducible constituent of both $\hat{\tau}^{H_1}$ and $\hat{\phi}_{H_1}$, since $[\hat{\phi}_{H_1},\tau^{H_1}] = [\hat{\phi}_{H_2},\tau^{H_2}] = 1$. In particular, since $ [\eta_U,\tau] \leq [\hat{\phi}_U,\tau] = 1 $ and $\tau$ is (by definition) $T$-invariant, we have that $\eta$ extends $\tau$. Suppose that $\eta^t \neq \eta$ for some $t \in T$; since both $\hat{\phi}$ and $\tau$ are $T$-invariant, both $\eta$ and $\eta^t$ are irreducible constituents of $\hat{\phi}_{H_1}$ lying over $\tau$, in contradiction with the fact that $[\hat{\phi}_{H_1},\tau^{H_1}] = 1$. Hence, $\eta$ is $T$-invariant and it follows that $H_1=H_2=H$, and $\hat{\tau}^{HM}=\hat{\phi}$ because they have the same degree. Now, since $\hat{\tau}$ is fully ramified in $T/H$ and $\hat{\phi}$ is fully ramified in $L/HM$, we have that $\hat{\tau}^T = e\gamma$ and $\hat{\phi}^L = e \psi$, with the constant $e$ defined by $e^2 = \abs{T:H} = \abs{L:HM}$. In particular, we have that $\gamma^L(1) = \abs{L:T}e\tau(1) = e \abs{M:U}\tau(1) = e\phi(1)=\psi(1)$ and, since
$$ e[\gamma^L,\psi] = [\gamma^L,\hat{\phi}^L] = [(\gamma^L)_{HM},\hat{\phi}] = e [\hat{\tau}^{HM},\hat{\phi}] = e > 0, $$
it follows that $\gamma^L=\psi$.

Suppose that $d \in \mathcal{H}$ fixes $\psi$, then it also fixes both $\hat{\phi}$ and $\phi$, since $\psi_M=e\phi$, and it follows from our hypothesis that $(U,\tau)^d = (U,\tau)^g$ for some $g \in M$. Hence, $x = d g^{\sminus 1}$ stabilizes $(U,\tau)$ and, therefore, it normalizes both $T$ and $H$. Since $\hat{\tau}$ is the only irreducible constituent of both $\hat{\phi}_H$ and $\tau^H$, and since $x$ fixes both $\hat{\phi}$ and $\tau$, we have that $x$ also fixes $\hat{\tau}$ and, thus, it also fixes $\gamma$. Hence, we have proved that $(T,\gamma)^x=(T,\gamma)$ and, therefore, $(T,\gamma)^d=(T,\gamma)^g$.

Finally, if $U^a=U$ and $\tau^a=\tau^{\sigma}$, then $N^a=N$, where $N = \no{L}{U}$, and $T^a = I_N(\tau^a) = I_N(\tau^{\sigma}) = T$. Moreover, since $H^a$ is maximal in $T$ such that $\tau^a=\tau^{\sigma}$ has a $T$-invariant extension to $H^a$, and $\tau^{\sigma}$ has a $T$-invariant extension if and only if so has $\tau$, we have that $H^a=H$. Notice that $H/U$ is an elementary abelian $p$-group. Define a function $\rho$ on $H/U$ such that $u^{\rho} = (u^k)^a$ for each $u \in H/U$, where $k$ was defined in the second paragraph of the lemma, and notice that it is an automorphism of $H/U$ of order $s$ coprime with $p$. If there is no element $x \in H/U$ of order $p$ such that $x^{a^{\sminus 1}} = x^k$, then $\ce{H/U}{\rho} = 1$.

%Now, since for every $\lambda \in \irr(H/U)$ and for every $x \in H/U$ we have that $\lambda^{\sigma}(x) = \lambda(x^k)$, we are justified in abusing the notation when we define the function $\rho$ on $\irr(H/U)$ such that $\lambda^{\rho} = (\lambda^{\sigma})^{a ^{\sminus 1}}$, since $\lambda^{\rho}(x) = \lambda(x^{\rho})$. 
Now, with a little abuse of notation, we shall define the function $\rho$ on $\irr(H)$ such that $\alpha^{\rho} = (\alpha^{\sigma})^{a ^{\sminus 1}}$, for every $\alpha \in \irr(V)$ and every $a$-invariant $V \leq H$. Notice that, since for every $\lambda \in \irr(H/U)$ and for every $x \in H/U$ we have that $\lambda^{\sigma}(x) = \lambda(x^k)$, then $\lambda^{\rho}(x) = \lambda(x^{\rho})$ and our abuse of notation is justified. In particular, we have that $\tau^{\rho}=\tau$ and, thus, $\hat{\tau}^{\rho} = \lambda \hat{\tau}$ for some $\lambda \in \irr(H/U)$, by Gallagher's theorem. Since, by coprime action, we have
$$ H/U = \ce{H/U}{\rho} \times [H/U,\rho], $$
then we can write
$$ \lambda = \mu (\nu^{\sminus 1})^{\rho} \nu, $$
for some $\mu,\nu \in \irr(H/U)$ with $\mu$ being fixed by $\rho$. Hence, by replacing $\hat{\tau}$ with $\nu \hat{\tau}$ we have that $\hat{\tau}^{\rho} = \mu \hat{\tau}$. It follows that $\hat{\tau}^{\rho^m} = \mu^m \hat{\tau}$ for every integer $m$ and, in particular, $\mu^s = 1_{H/U}$; since $o(\mu)=p$ is coprime with $s$, it follows that $\mu = 1_{H/U}$ and $\hat{\tau}$ is fixed by $\rho$. Moreover, if $\ce{H/U}{\rho} = 1$, then $\hat{\tau}$ is the only irreducible constituent of $\tau^H$ to be fixed by $\rho$ and, thus, $\tau$ and $\hat{\tau}$ have the same field of values. Finally, since $\hat{\tau}^T = e \gamma$, we have that $\gamma^{\rho} = \gamma$ and that $\tau$ and $\gamma$ have the same field of values. Now the claim is proved.

It follows from our claim that the quadruple $(L,\psi,T,\gamma)$ verifies the hypotheses of our lemma. Hence, by induction we have that there exists a subgroup $T \leq W \leq G$ and a character $\theta \in \irr(W \smid \gamma)$ such that $\theta(1)o(\theta)$ is a power of $p$ and $\theta^G = \chi \in \irr(G \smid \psi)$. In particular, $\theta$ also lies over $\tau$ and $\chi$ over $\phi$. %If $c \in \mathcal{H}$ fixes $\chi$, then $\psi^c = \psi^g$ for some $g \in G$ and, thus, $x=c g^{\sminus 1}$ fixes $\psi$ and the claim guarantees that $(T,\gamma)^x = (T,\gamma)^h$ for some $h \in L$. It follows that $(T,\gamma)^c = (T,\gamma)^{hg}$ and, by induction, we have that $(W,\theta)^c$ and $(W,\theta)$ are conjugate in $G$.
Moreover, the induction hypothesis guarantees that $(W,\theta)^c$ and $(W,\theta)$ are conjugate in $G$ for every $c \in \mathcal{H}$ fixing $\chi$. If $U^a = U$ and $\tau^a=\tau^{\sigma}$, then the claim guarantees that $T^a=T$ and $\gamma^a=\gamma^{\sigma}$ and it follows by induction that also $W^a=W$, $\theta^a=\theta^{\sigma}$ and $\chi^a=\chi^{\sigma}$. Finally, if $\tau$ has values in $\cycl{p}$ and no element $x \in G/E$ of order $p$ is such that $x^{a^{\sminus 1}} = x^k$ for every $a$-invariant subgroup $E$ with $M \leq E \lhd G$, then $\gamma$ has values in $\cycl{p}$ and, by induction, the same is true for $\theta$ and for $\chi$. Hence, the lemma is proved.
\end{proof}

We now list some properties of rational elements, which can be found proved in \cite{Navarro-Tiep:Rational_irreducible_characters}, and they are reported also in \cite[Lemma~2.4]{Tiep-TongViet:Odd-degree_rational}.

\begin{lemma}
\label{lemma:rational_elements}
Let $G$ be a finite group and let $N \lhd G$.
\begin{enumerate}[(a)]
\item If $x \in G$ is rational in $G$, then $xN \in G/N$ is rational in $G/N$.
\item If $g \in H \leq G$ is rational in $H$, then $g$ is also rational in $G$.
\item Assume that $x \in G$ has prime order $p$. Then, $x$ is rational in $G$ if and only if there exists a $p'$-element $g \in G$ such that $x^g=x^t$, where $t$ (mod $p$) is any generator of the multiplicative group $\mathbb{Z}_p^{\times}$.
\item If $x \in G$ is rational, then every power of $x$ is also rational; moreover, its $\pi$-part $x_{\pi}$ is also rational for every set of primes $\pi$.
\item If $x \in G$, $\op{gdc}(o(x),\abs{N})=1$ and $xN$ is rational in $G/N$, then $x$ is rational in $G$.
\item If $G/N$ has a rational element of prime order $p$, then $G$ has also a rational element of order $p$.
\end{enumerate}
\end{lemma}

\begin{proof}
All these statements can be found in \cite[Lemmas 5.1 and 5.2]{Navarro-Tiep:Rational_irreducible_characters}.
\end{proof}

It is proved in \cite[Theorem~C]{Tiep-TongViet:Odd-degree_rational} that, if $G$ is an almost-simple group with socle $S \ncong \op{PSL}_2(3^f)$, for some odd integer $f \geq 3$, then $G$ has a rational irreducible character of even degree. We now prove that it may happen also when $S \cong \op{PSL}_2(3^f)$, provided that $G \neq S$.

\begin{proposition}
\label{proposition:almost-simple}
Let $S \cong \op{PSL}_2(3^f)$, for some odd integer $f$, and let $S < G \leq \op{Aut}(S)$, where we identify $S$ with $\op{Inn}(S)$.
\begin{enumerate}[(a)]
\item $G$ has a non-trivial rational element $x \in S$ of odd order if and only if $G$ has an irreducible rational character of even degree not containing $S$ in its kernel.
\item If there exists $S < H \leq G$ such that $H \cong \op{PGL}_2(3^f)$, then $G$ has an irreducible rational character of even degree not containing $S$ in its kernel, and a non-trivial rational element of odd order $x \in S$.
\end{enumerate}
\end{proposition}

\begin{proof}
Suppose first that there exists $S < H \leq G$ such that $H \cong \op{PGL}_2(3^f)$, then we can see from \cite[Theorem~38.1]{Dornhoff:Group_Representation_TheoryA} and from \cite[Table~III]{Steinberg:The_representations_of} that there exists $\phi \in \irr(S)$ of degree $\phi(1) = \frac{3^f - 1}{2}$ such that $\phi^H=\eta \in \irr(H)$ is a $G$-invariant rational character of even degree $\eta(1) = 3^f - 1$. Since $G/H$ is a cyclic group of odd order dividing $f$, it follows from a repeated application of \cite[Theorem~6.30]{Isaacs} that $\eta$ has a rational extension $\psi \in \irr(G)$, of even degree. Moreover, we can notice that there exists $x = \begin{psmallmatrix}1 & 0\\ 1 & 1\end{psmallmatrix} \in \op{GL}_2(3^f)$, of order $3$, such that $x^h = x^{\sminus 1}$, where $h = \begin{psmallmatrix}1 & 0\\ 0 & -1\end{psmallmatrix}$. Since $x \notin \ze{\op{GL}_2(3^f)}$, the element $\hat{x} \in \op{PGL}_2(3^f)$ corresponding to $x$ under the canonical epimorphism is a rational element of order $3$. It follows that $H$ has a rational element of order $3$ and, thus, so does $G$.

We can now assume that $G/S$ is generated by a field automorphism $\tau$ of odd order dividing $f$, and we let $\psi \in \irr(S)$ be the Steinberg character of $S$ and $\hat{\psi}$ be its unique rational extension to $G$. We claim that, if $\chi \in \irr(G)$ is rational such that $\chi \notin \{1_G, \hat{\psi} \}$, then $\chi(1)$ is even and does not contains $S$  in its kernel. Suppose $\chi(1)$ is odd and let $\phi$ be an irreducible constituent of $\chi_S$, then $\phi \neq \psi, 1_S$ and $\phi(1)$ is odd, and it follows from \cite[Theorem~38.1]{Dornhoff:Group_Representation_TheoryA} that $\phi(1)=\frac{3^f - 1}{2}$ and $\phi$ is not rational. Since $\chi$ is rational, we have that $\phi$ is not the only irreducible constituent of $\chi_S$ and, since $\phi$ and $\bar{\phi}$ are the only two irreducible characters of $S$ of their degree, we have that $\chi(1) = e(\phi(1) + \bar{\phi}(1)) = 2e\phi(1)$, for some positive integer $e$, is even.

We now claim that, if $x \in G$ is rational of prime power order and $o(x) \notin \{1,2\}$, then $x \in S$ and $o(x)$ is odd. Suppose $o(x)$ is even; since $G = \groupgen{\tau} \ltimes S$ and there are no rational elements in $\groupgen{\tau}$, we have that $x \in S$. Now, let $K \cong \op{SL}_2(q)$, where $q=3^f$, so that $K/\ze{K} \cong S$. We can see from \cite[Theorem~38.1]{Dornhoff:Group_Representation_TheoryA} that there exists $a,b \in K$, with $o(a)=q-1$ and $o(b)=q+1$, such that $\ze{K} \leq \groupgen{a} \cap \groupgen{b}$, and every element of even order in $S \cong K/\ze{K}$ is conjugate to a power of either $a\ze{K}$ or $b\ze{K}$. Since $\abs{\ze{K}}=2$, $4 \nmid q - 1$ and $8 \nmid q + 1$, we have that $o(a\ze{K})$ is odd and $o(b\ze{K})_2 = 2$ and it follows our claim.

Now, \cite[Lemma~11.4]{Navarro-Tiep:Rational_irreducible_characters} guarantees that $G$ has more than two rational irreducible characters if and only if it has more than two rational conjugacy classes. Hence, if $x \in S$ of odd order is rational in $G$, then there exists $\chi \in \irr(G)$ rational such that $\chi \notin \{1_G, \hat{\psi} \}$, and it follows from our first claim that $\chi(1)$ is even. On the other hand, if $G$ has a rational irreducible character of even degree, then there exists $x \in G$ rational of prime power order (by Lemma~\ref{lemma:rational_elements}(d)) which does not belong to the (unique) conjugacy class of elements of order $2$ in $G$, and it follows from our second claim that $x \in S$ has odd order.
%Because of Lemma~\ref{lemma:rational_elements}(d), we can assume that $o(x)$ is a power of some odd prime number $p$. Since $G = \groupgen{\tau} \ltimes S$, by Lemma~\ref{lemma:rational_elements}(g) we can assume that $x \in S$. Suppose first that $p=3$, then $o(x)=3$ and there exists $g \in G$ such that $x^g = x^{\sminus 1}$. Since $g$ normalizes $\groupgen{x}$, we can assume that $g$ has order $2$ and it follows that $g \in S$, since $\abs{G/S}$ is odd. However, the elements of $S$ of order $3$ are not rational in $S$ (see for instance \cite[Theorem~38.1]{Dornhoff:Group_Representation_TheoryA})
\end{proof}

The following proposition is a rather direct consequence of \cite[Theorem~C]{Tiep-TongViet:Odd-degree_rational} and of our Proposition~\ref{proposition:almost-simple}, and it allows us to determine which non-abelian simple groups can appear as composition factors in a group where all the rational irreducible characters have odd degree.

\begin{proposition}
\label{proposition:involved_groups}
Let $G$ be a finite group and let $M,L \lhd G$, with $L < M$, such that $M/L \cong S_1 \times ... \times S_n$ is the direct product of some isomorphic non-abelian simple groups $S_i \cong S$, for each $i=1,...,n$.
\begin{enumerate}[(a)]
\item If $S \ncong \op{PSL}_2(3^f)$ for any odd integer $f \geq 3$, then $G$ has a rational irreducible character of even degree and a rational element of odd order.
\item If $S \cong \op{PSL}_2(3^f)$, for some odd integer $f \geq 3$, and there exists $xL \in M/L$ of odd order which is rational in $G/L$, then $G$ has a rational irreducible character of even degree.
\end{enumerate}
\end{proposition}

\begin{proof}
Working by induction on $\abs{G}$, we can assume that $L=1$ and $M$ is a minimal normal subgroup. Let $N=\no{G}{S_1}$ and $C=\ce{G}{S_1} \leq N$, then $N/C$ is isomorphic to a subgroup of $\op{Aut}(S_1)$ containing $S_1$. If $S \ncong \op{PSL}_2(3^f)$, as in case (a), then it follows from \cite[Theorem~C]{Tiep-TongViet:Odd-degree_rational} that there exists $\hat{\psi} \in \irr(N/C)$ rational, of even degree and not containing $S_1C/C$ in its kernel. If $S \cong \op{PSL}_2(3^f)$ and there exists $1 \neq x \in S_1$ rational in $N$ of odd order, then $xC \in S_1C/C$ is rational in $N/C$ and it follows from Proposition~\ref{proposition:almost-simple} that also in this case there exists $\hat{\psi} \in \irr(N/C)$ rational, of even degree and not containing $S_1C/C$ in its kernel. Hence, in both of these cases, there exists $\psi \in \irr(N)$, containing $C$ in its kernel but not $S_1$, such that $\psi$ is a rational character of even degree. If $\delta \in \irr(S_1)$ is an irreducible constituent of $\psi_{S_1}$, and $\phi = \delta \times 1_{S_2} \times ... \times 1_{S_n} \in \irr(M)$, then $T = I_G(\phi) \leq N$ and there exists $\theta \in \irr(T \smid \phi)$ such that $\theta^N=\psi$. Hence, $\psi^G=(\theta^N)^G$ is an irreducible rational character of even degree.

Suppose now that no non-trivial element of $S_1$ of odd order is rational in $N$, and let $1 \neq x = x_1 \times ... \times x_n \in M$ of odd order be rational in $G$. Then, in particular, there exists $g \in G$, of order a power of $2$, such that $x^g = x^{\sminus 1}$. Without loss of generality, we can assume that $x_1 \neq 1$, hence, there exists $i \in \{1,...,n\}$ such that ${x_1}^{\sminus 1} = {x_i}^g$. Since we are assuming that $\abs{N/C}$ is odd, because of Proposition~\ref{proposition:almost-simple}(b), and since ${x_1}^{\sminus 1} \neq x_1$ because $o(x_1)$ is odd, we have that $g \notin N$. Let $m$ be minimal such that $g^m \in N$, and notice that $g^m \in C$ and $m > 1$ is a power of $2$. In particular, if we replace $g$ with $g^{m/2}$, we still have that $g \notin N$ and we also have that $g^2 \in C$. Let $j \in \{1,...,n\}$ such that ${S_1}^g = S_j$, then ${S_j}^g = {S_1}^{g^2} = S_1$ and $j \neq 1$. Without loss of generality, we may assume that $j=2$.
%Moreover, since $x^{g^2} = x$, we have that $x_i = {x_i}^{g^2} = ({x_1}^{\sminus 1})^g$ and, in particular, it follows that ${S_1}^g = S_i$ and ${S_i}^g = S_1$. Since we are assuming that $\abs{N/C}$ is odd, because of Proposition~\ref{proposition:almost-simple}(b), we have that $g^2 \in C$ and $g \notin N$, hence $i \neq 1$. Without loss of generality, we may assume that $i=2$. 
Consider now $\phi_1 \in \irr(S_1)$ of degree $\phi_1(1)=\frac{3^f - 1}{2}$ and notice that $\phi_1$ has values in $\Q_3$, $\phi_1 + \overline{\phi_1}$ is rational, and $\phi_1$ has a canonically defined extension to $N$. In fact, $\phi_1$ extends to $\theta_1 = 1_C \times \phi_1 \in \irr(C \times S_1)$, which in turn has a canonical extension $ \theta_2 \in \irr(Q)$, where $Q/C \times S_1 \in \syl{3}{N/C \times S_1}$, since $3 \nmid \theta_1(1)o(\theta_1)$; finally, since $\theta_2$ has values in $\Q_3$ and $(6,\abs{N/Q}) = 1$, then it has a unique extension to $N$ having values in $\Q_3$, by \cite[Theorem~6.30]{Isaacs}. Now, let $\phi_2 = {\phi_1}^g \in \irr(S_2)$ and let $\phi = \phi_1 \times \overline{\phi_2} \times 1_{S_3} \times ... \times 1_{S_n} \in \irr(M)$; notice that $\phi^g = \overline{\phi}$, that $\phi$ has a canonical extension to $T = I_G(\phi) = \no{G}{S_1} \cap \no{G}{S_2}$, and that $g$ normalizes $T$. Let $\theta \in \irr(T)$ be this canonical extension of $\phi$, then we have that $\theta$ and $\phi$ have the same field of values, $\theta^g = \overline{\theta}$, and, therefore, $\theta + \theta^g$ is rational. As a consequence, if we consider $K = \no{G}{S_1 \times S_2}$, we have that $\abs{K:T}=2$, $\psi = \theta^K \in \irr(K)$ is a rational character, and $\psi(1)=2\phi(1)$. It follows that $\psi^G = (\theta^K)^G$ is an irreducible rational character of even degree.

%, and it follows that $g \notin C$. Moreover, we know from Proposition~\ref{proposition:almost-simple}(b) that $\abs{N/C}$ is odd, since otherwise $N/C$ would have an odd-order rational element. It follows that $g \notin N$

Finally, it follows from \cite[Theorem~11.1]{Navarro-Tiep:Rational_irreducible_characters} that, if $S \ncong \op{PSL}_2(3^f)$ for any odd integer $f \geq 3$, then $M/L$ contains a rational element of odd prime order. It follows from Lemma~\ref{lemma:rational_elements}(f) that also $M$ contains a rational element of odd prime order and, thus, so does $G$.
\end{proof}

We are now ready to prove our main theorem. Our proof will use many of the ideas of the proof of \cite[Theorem~A]{Tiep-TongViet:Odd-degree_rational}.

\begin{proof}[Proof of the Main Theorem]
Suppose, by contradiction, that $G$ is a counterexample of minimal order, and let $M \lhd G$ be a minimal normal subgroup. Since $\irr(G/M) \subseteq \irr(G)$, by minimality of $\abs{G}$ we have that $G/M$ does not contain any non-trivial rational element of odd order. Let $x \in G$ be a rational element of odd order, then it follows from Lemma~\ref{lemma:rational_elements}(a) that $xM=M$ and thus, $x \in M$. Moreover, Lemma~\ref{lemma:rational_elements}(d) allows us to assume that $x$ has $p$-power order, for some odd prime $p$.

Suppose that $M = S_1 \times ... \times S_n$ is the direct product of some non-abelian simple groups, then it follows from Proposition~\ref{proposition:involved_groups}(a) that, for each $i=1,...,n$, $S_i \cong \op{PSL}_2(3^f)$ for some odd integer $f \geq 3$, and then it follows from Proposition~\ref{proposition:involved_groups}(b) that $G$ has a rational irreducible character of even degree, in contradiction with our hypothesis.

%Moreover, Lemma~\ref{lemma:rational_elements}(g) allows us to assume, without loss of generality, that $x \in S_1=S$, and that $x$ is rational in $\no{G}{S}$. Let $m$ be a positive integer such that $p \nmid m$ and $m$ has order $2$ in $\mathbb{Z}_p^{\times}$. Since $x$ is rational in $\no{G}{S}$ but not in $S$, there exists $1 \neq g \in \no{G}{S}/S\ce{G}{S}$ such that $x^g = x^m$ and, thus, $x^{g^2}=x$. In particular, it follows that $g$ has even order and, therefore, $\no{G}{S}/S\ce{G}{S}$ is even. Since $S \cong \op{PSL}_2(3^f)$ with $f$ odd, there exists $S\ce{G}{S} < H \leq \no{G}{S}$ such that $H/\ce{G}{S} \cong \op{PGL}_2(3^f)$ and it follows from Proposition~\ref{proposition:involved_groups} that $G$ has a rational irreducible character of even order, in contradiction with our hypothesis.

Assume now that $M$ is an elementary abelian $p$-group, then $x$ has order $p$ and, by Lemma~\ref{lemma:rational_elements}(c), there exists a $p'$-element $1 \neq g \in G$ such that $x^{g^{\sminus 1}}=x^t$, where $t$ is a generator for $\mathbb{Z}_p^{\times}$. In particular, by duality, there exists $1_M \neq \lambda \in \irr(M)$ is such that $\lambda^g=\lambda^t$. Let $\sigma \in \gal(\Q_{\abs{G}}/\Q)$ of order $p-1$, fixing $p'$th-roots of unity and such that $\xi^{\sigma}=\xi^t$, where $\xi$ is a primitive $p$th-root of unity. We have that $\lambda^{\sigma}=\lambda^t=\lambda^g$. Moreover, we have that, for any subgroup $E$ with $M \leq E \lhd G$ and for any $y \in G/E$ of order $p$, $y^{g^{\sminus 1}} \neq y^t$, since otherwise it would follow from Lemma~\ref{lemma:rational_elements}(c) that $y$ is a rational element of $G/E$ of odd order, in contradiction with the minimality of $\abs{G}$. We can also notice that, by Proposition~\ref{proposition:involved_groups}(a), every non-abelian simple group involved in $G$ is isomorphic to $\op{PSL}_2(3^f)$, for some odd integer $f$. Now, every non-solvable subgroup of $\op{PSL}_2(3^f)$ is isomorphic to $\op{PSL}_2(3^e)$, for some integer $e \mid f$, by \cite[II.8.27~Hauptsatz]{Huppert:Endliche_Gruppen}, and $\abs{M(\op{PSL}_2(q))} = 2$ for every odd prime power $q \neq 9$. Hence, we have that $p \nmid \abs{M(S)}$ for every non-abelian composition factor $S$ of $H/M$, for every $M \leq H \leq G$, whenever $p$ is an odd prime. Hence, we can apply our Lemma~\ref{lemma:existing_over}, with $(M,\lambda)$ playing the part of both $(M,\phi)$ and $(U,\tau)$, and we have that there exists $\chi \in \irr(G \smid \lambda)$ having values in $\cycl{p}$ and such that $\chi^{\sigma}=\chi^g=\chi$.

Since $\chi$ has values in $\cycl{p}$ and $\sigma$ is a generator for $\gal(\Q_p/\Q)$, we have that $\chi$ is rational and, by hypothesis, $\chi(1)$ is odd. It follows from \cite[Lemma~2.1]{Tiep-TongViet:Odd-degree_rational} that $\lambda$ is a real-valued character and, since $p$ is odd, we have that $\lambda = 1_M$, in contradiction with our choice.
\end{proof}

\printbibliography[heading=bibintoc]

\end{document}